\documentclass[11pt]{amsart}

\usepackage{amssymb, amsthm, amsmath}
\usepackage{graphicx}
\usepackage[small]{caption}
\usepackage{subcaption}
\usepackage{epsfig}
\usepackage{tikz, pgfplots, float}

\newcommand{\later}[1]{}
\newcommand{\old}[1]{}

\usepackage{amsfonts}

\usepackage[utf8]{inputenc}
\usepackage{fullpage}
\usepackage{framed}

\usepackage{enumerate}
\usepackage{url}
\usepackage{hyperref}

\newtheorem{theorem}{Theorem}
\newtheorem{lemma}[theorem]{Lemma}
\newtheorem{claim}[theorem]{Claim}

\newtheorem{proposition}[theorem]{Proposition}
\newtheorem{corollary}[theorem]{Corollary}

\newtheorem{question}[theorem]{Question}
\newtheorem{conjecture}[theorem]{Conjecture}

\newcommand{\eg}{{e.g.}}

\newcommand{\eps}{\varepsilon}

\newcommand{\dist}{{\rm dist}}

\newcommand{\ZZ}{\mathbb{Z}} %  set of integer numbers
\usepackage{tikz}
\tikzstyle{vtx} = [circle, fill, inner sep=0.7]

\title{Subset selection problems in planar point sets}

\author{J\'ozsef Balogh}
\address{J\'ozsef Balogh \newline University of Illinois at Urbana-Champaign, USA}
\email{jobal@illinois.edu}

\author{Felix Christian Clemen}
\address{Felix Christian Clemen \newline Institute for Basic Science, South Korea}
\email{felix.clemen@ibs.re.kr}

\author{Adrian Dumitrescu}
\address{Adrian Dumitrescu \newline Algoresearch L.L.C., Milwaukee, WI, USA \newline Alfr\'ed R\'enyi Institute of Mathematics, Budapest, Hungary \newline Research Institute of the University of Bucharest, Romania}
\email{ad.dumitrescu@algoresearch.org}

\author{Dingyuan Liu}
\address{Dingyuan Liu \newline Karlsruhe Institute of Technology, Germany}
\email{liu@mathe.berlin}

\begin{document}

\maketitle

\begin{abstract}
Given a finite point set satisfying condition $\mathcal{A}$, the subset selection problem asks, how large of a subset satisfying condition $\mathcal{B}$ can be extracted? Problems of this nature are notoriously difficult and have attracted extensive study. In this paper, we make progress on three instances of subset selection problems in planar point sets. Let $n,s\in\mathbb{N}$ with $n\geq s$, and let $P\subseteq\mathbb{R}^2$ be a set of $n$ points, where at most $s$ points lie on the same line. 

Firstly, we select a general position subset of $P$, i.e., a subset containing no $3$ points on the same line. This problem was proposed by Erd\H{o}s under the regime when $s$ is a constant. For $s$ being non-constant, we give new lower and upper bounds on the maximum size of such a subset. In particular, we show that in the worst case such a set can have size at most $O(n^{5/6+o(1)}/\sqrt{s})$ when $3\leq s\leq n^{1/3}$ and $O(n/s)$ when $n^{1/3}\leq s\leq n$. 

Secondly, we select a monotone general position subset of $P$, that is, a subset in general position where the points are ordered from left to right and their $y$-coordinates are either non-decreasing or non-increasing. We present bounds on the maximum size of such a subset. In particular, when $s=\Omega(\sqrt{n})$, our upper and lower bounds differ at most by a logarithmic factor. 

Lastly, we select a subset of $P$ with pairwise distinct slopes. This problem was initially studied by Erd\H{o}s, Graham, Ruzsa, and Taylor on the grid. We show that for $s=O(\sqrt{n})$ such a subset of size $\Omega((n/\log{s})^{1/3})$ can always be found in $P$. When $s=\Theta(\sqrt{n})$, this matches a lower bound given by Zhang on the grid. As for the upper bound, we show that in the worst case such a subset has size at most $O(\sqrt{n})$ for $2\leq s\leq n^{3/8}$ and $O((n/s)^{4/5})$ for $n^{3/8}\leq s=O(\sqrt{n})$. 

The proofs use a wide range of tools such as incidence geometry, probabilistic methods, the hypergraph container method, and additive combinatorics.
\end{abstract}

\section{Introduction} \label{sec:intro}

For an integer $k\geq3$, a \emph{collinear $k$-tuple} is a set of $k$ distinct points lying on the same line.
A planar point set is said to be in \emph{general position} if it contains no collinear triple.
The problem of selecting a large general position subset from an $n \times n$ integer grid, also known as the \emph{no-three-in-line problem}, goes back more than $100$ years, see for examples~\cite{Dud1917} and~\cite[Ch.~10]{BMP05}. A pigeonhole principle argument shows that one can select at most $2n$ points. Dudeney~\cite{Dud1917} asked whether one can always select $2n$ points and this question remains open.

The no-three-in-line problem is a typical instance of the \emph{subset selection problem}. In general, the subset selection problem asks: Given a finite point set satisfying condition $\mathcal{A}$, how large of a subset satisfying condition $\mathcal{B}$ can be extracted?

Such problems are known for their intrinsic difficulty and therefore often serve as important benchmark tasks
in machine learning. For example, recent AI-assisted methods~\cite{CEWW24,GGTW25} have been used to search for
a large subset of an $n\times n$ grid that spans no isosceles triangle, as well as a large subset of an
$n\times n\times n$ grid containing no $5$ points on a common sphere. In this paper, we focus on three subset
selection problems in planar point sets, two of which have received considerable attention in extremal combinatorics
and discrete geometry (see Sections~\ref{general} and~\ref{slopes}), while the remaining one is a new variation (see Section~\ref{monotone_general}).

\subsection{Selecting a general position subset}
\label{general}

A natural generalization of the no-three-in-line problem is to select a
large general position subset from an arbitrary planar point set satisfying 
a certain collinearity condition. Let $n,s\in\mathbb{N}$ with $n\geq{s}\geq2$. The
function $f(n,s)$ is defined as the largest integer, such that every set
$P\subseteq\mathbb{R}^{2}$ of $n$ points with at most $s$ points collinear,
i.e., without collinear $(s+1)$-tuples, contains a general position subset of size $f(n,s)$. It follows from the definition that $f(n,2)=n$. Erd\H{o}s~\cite{Erd1986} studied the first non-trivial case $f(n,3)$ and showed that $f(n,3)=\Omega(\sqrt{n})$ by a greedy selection. Later F\"{u}redi~\cite{Fur1991} proved that $f(n,3)=\Omega(\sqrt{n\log{n}})$ using the lower bound on the independence number of partial Steiner triple systems~\cite{PR1986}. Employing a similar idea, Lefmann~\cite{Lef10} generalized F\"{u}redi's lower bound to all $f(n,s)$ with fixed $s\geq3$.

On the other hand, F\"{u}redi~\cite{Fur1991} constructed a set
$P\subseteq\mathbb{R}^{2}$ of $n$ points containing no collinear $4$-tuple and
used the density Hales--Jewett theorem~\cite{FK1989} to show that $P$ contains at
most $o(n)$ points in general position, namely, $f(n,3)=o(n)$. Balogh and
Solymosi~\cite{BS18} proved that a random $n$-element subset of a large
$3$-dimensional grid contains no collinear $4$-tuple and at most $n^{5/6+o(1)}$
points in general position with high probability. After projecting this subset
into the plane they obtained a better upper bound construction, implying that
$f(n,3)\leq{n^{5/6+o(1)}}$. Note that $f(n,s)$ is monotone in $s$, namely,
$f(n,s)\leq{f(n,s')}$ holds for all $s\geq{s'}$. Thus the aforementioned
results\footnote{Further discussion on these aspects can be found
  in~\cite[Ch.~9]{Epp18}.} established that 
\begin{equation}
\label{gap}
\Omega(\sqrt{n\log{n}})=f(n,s)\leq{f(n,3)}\leq{n^{5/6+o(1)}}\quad\text{for any fixed $s\geq3$}.
\end{equation}

Payne and Wood~\cite{PW13} were the first to study the function $f(n,s)$ when $s$ is a function of $n$.
They showed that
\begin{equation}
\label{pay}
f(n,s)=
\begin{cases}
\Omega(\sqrt{n\log_{s}{n}}) & \text{for $3\leq s \leq{n^{1/2-c}}$ with any fixed constant $c>0$,}\\
\noalign{\vskip9pt}
\Omega(\sqrt{n/\log{n}}) & \text{for $n^{1/2-o(1)}\leq{s}\ll\sqrt{n\log{n}}$,}\\
\noalign{\vskip9pt}
\Omega\left(n/s\right) & \text{for $s=\Omega(\sqrt{n\log{n}})$.}
\end{cases}
\end{equation}

We present new upper bounds on $f(n,s)$ in the non-constant regime of $s$.
\begin{theorem}
\label{law}
We have
\begin{equation}
f(n,s)= 
\begin{cases}
O(n^{5/6+o(1)}/\sqrt{s}) & \text{for } 3 \leq s \leq n^{1/3}, \\
\noalign{\vskip9pt}
O(n/s) & \text{for } n^{1/3} \leq s \leq n.
\end{cases}
\end{equation}
In particular, $f(n,s)= \Theta(n/s)$ when $s=\Omega(\sqrt{n\log{n}})$.
\end{theorem}
Our proof of Theorem~\ref{law} in the case $s\leq n^{1/3}$ follows the approach of Balogh and
Solymosi~\cite{BS18} on the upper bound on $f(n,3)$, using the hypergraph container method. This method has
been widely employed in geometric settings, see, e.g.,~\cite{CLNZ23,RN24,RN25,SZ23}. The main novelty in our
proof is a balanced supersaturation result, which allows us to iteratively apply the hypergraph lemma with
improved parameters and thereby establish the desired upper bound over a wider range of $s$. Indeed, the
supersaturation result of Balogh and Solymosi~\cite{BS18} (see also Keller and Smorodinsky~\cite{KS21}) yields
the upper bound $f(n,s)=O(n^{5/6+o(1)}/\sqrt{s})$ only for $s\leq n^{1/6}$, leaving a weaker upper bound on
$f(n,s)$ in the range $n^{1/6}<s\leq n^{1/3}$. 

A result of Hajnal and Szemer\'{e}di~\cite[Theorem 3]{HS18} yields an improved lower bound
$f(n,s)=\Omega\left(\sqrt{n\log\log{n}/\log{n}}\right)$ for $n^{1/2-o(1)}\leq
s=O\left(\sqrt{n\log\log{n}/\log{n}}\right)$. Here we extend the range of $s$ where a similar lower bound
holds. 
\begin{proposition}
\label{new}
Let $n,s\in\mathbb{N}$ with $3\leq s\ll\sqrt{n\log{n}}$. Then
\begin{equation}
\label{oak}
f(n,s)=\Omega\left(\sqrt{\frac{n\log{\frac{n\log{s}}{s^2}}}{\log{s}}}\right).
\end{equation}
In particular, for $n^{1/2-o(1)}\leq s = O(\sqrt{n\log\log{n}})$ we have
$f(n,s)=\Omega\left(\sqrt{\frac{n\log\log{n}}{\log{n}}}\right)$. 
\end{proposition}
 Note that our lower bound in~\eqref{oak} is consistent
with~\eqref{pay} when $s\leq{n^{1/2-c}}$ for any fixed constant $c>0$ and better in the critical range
$n^{1/2-o(1)}\leq{s}\ll\sqrt{n\log{n}}$.

\subsection{Selecting a monotone general position subset}
\label{monotone_general}

According to a well-known theorem due to Erd\H{o}s and Szekeres~\cite{ES35}, for
any $\ell,m\geq1$, every sequence $(x_i)_{i=1}^{n}$ of real numbers with $n>\ell
m$ contains a non-decreasing subsequence of length larger than $\ell$, or a
non-increasing subsequence of length larger than $m$. For a point
$v\in\mathbb{R}^2$ denote by $v_x$ and $v_y$ the $x$-coordinate and $y$-coordinate of
$v$, respectively. We say that a point set $P\subseteq\mathbb{R}^2$ is
\emph{monotone} if one of the following holds: 
\begin{enumerate}[(i)]
\item $\forall\,u,v\in{P}:\,u_x<v_x \implies u_y \leq v_y$;
\item $\forall\,u,v\in{P}:\,u_x<v_x \implies u_y \geq v_y$.
\end{enumerate}
See Figure~\ref{fig:grid1} for an illustration of a monotone general position subset. A simple corollary of
the Erd\H{o}s--Szekeres theorem states that any set 
$P\subseteq\mathbb{R}^2$ of $n$ points contains a monotone subset of size at
least $\sqrt{n}$. Let $g(n,s)$ be the largest integer, such that any set
$P\subseteq\mathbb{R}^{2}$ of $n$ points with at most~$s$ collinear points
contains a monotone general position subset of size $g(n,s)$. It is not difficult to show that
$g(n,2)=\Theta(\sqrt{n})$. We are mainly interested in the following question: 

\begin{question} \label{q:mon-gp}
Let $n,s\in\mathbb{N}$ with $n\geq s \geq 3$. What is the growth order of $g(n,s)$?
\end{question}

We give a partial answer to this question as follows.

\begin{theorem} \label{thm:mon-gp}
Let $n,s\in\mathbb{N}$ with $n\geq s \geq 3$. Then
\begin{equation}
g(n,s)=
\begin{cases}
\Omega\left(\left(\frac{n\log{\frac{n\log{s}}{s^2}}}{\log{s}}\right)^{1/4}\right) & \text{for $3\leq
  s\ll\sqrt{n\log{n}}$,}\\ 
\noalign{\vskip9pt}
\Omega\left(\sqrt{n/s}\right) & \text{for $s=\Omega(\sqrt{n\log{n}})$.}
\end{cases}
\end{equation}
On the other hand, we have $g(n,s)=O(\sqrt{n/s})$.
\end{theorem}

It is worth noting that the gap between the bounds in Theorem~\ref{thm:mon-gp}
is large for small $s$ but almost vanishes when $s$ is large. In particular, when $s=\Omega(\sqrt{n})$, the
upper and lower bounds differ only by a logarithmic factor, and when $s=\Omega(\sqrt{n\log{n}})$ we have
$g(n,s)=\Theta(\sqrt{n/s})$. 
Additionally, for the $n$ points in a $\sqrt{n} \times \sqrt{n}$ integer
grid\footnote{Throughout this paper, whenever we refer to the $\sqrt{n} \times
  \sqrt{n}$ grid 
we assume that $n$ is a perfect square. One could work with a $t \times n/t$ grid instead, where
$t=\Theta(\sqrt{n})$ and $t$ divides $n$, but the results would be very similar.}, we deduce better bounds.

\begin{theorem} \label{thm:mon-gp-grid}
Let $G$ be a $\sqrt{n} \times \sqrt{n}$ grid of $n$ points. 
Then the largest size of a monotone general position subset of $G$ is at least
$\Omega\left((n/\log{n})^{2/5}\right)$ and at most $o(n^{1/2})$.  
\end{theorem}
The construction obtaining the lower bound in Theorem~\ref{thm:mon-gp-grid} is a random subset taken from the
intersection of the grid and a sector of an annulus with a carefully chosen radial width.   

\begin{figure}[htbp]
\centering
\begin{tikzpicture}[scale=0.48]

  \tikzset{
    gridline/.style={line width=0.35pt},
    dot/.style={circle,fill=black,inner sep=1.35pt}
  }

  %------------------------------------------------
  % Left: Dudeney's maximum-sized general
  % position subset in the 8 x 8 grid
  %------------------------------------------------
  \begin{scope}
    \draw[gridline] (0,0) grid (7,7);

    \foreach \x/\y in {
      2/7,6/7,
      2/6,4/6,
      0/5,1/5,
      4/4,6/4,
      1/3,3/3,
      5/2,7/2,
      0/1,3/1,
      5/0,7/0
    }{
      \node[dot] at (\x,\y) {};
    }
  \end{scope}

  %------------------------------------------------
  % Right: monotone general position subset
  %------------------------------------------------
  \begin{scope}[xshift=10cm]
    \draw[gridline] (0,0) grid (7,7);

    \foreach \x/\y in {
      6/7,7/7,
      4/6,5/6,
      4/4,
      2/1,3/1,
      0/0,1/0
    }{
      \node[dot] at (\x,\y) {};
    }
  \end{scope}

\end{tikzpicture}
\caption{Left: Dudeney's maximum-sized general position subset in the $8 \times 8$ grid.\\
Right: a monotone general position subset in the $8 \times 8$ grid.}
\vspace{-1em}
\label{fig:grid1}
\end{figure}
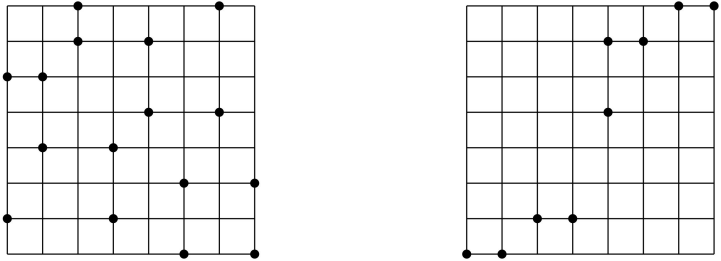

\subsection{Selecting a subset with pairwise distinct slopes}
\label{slopes}

A strengthening of the general position requirement is that of \emph{pairwise
  distinct slopes}. Indeed every point set with pairwise distinct slopes has no
collinear triple. 

For the $\sqrt{n} \times \sqrt{n}$ grid $G$, Erd\H{o}s, Graham, Ruzsa, and Taylor~\cite{EGRT92}
showed that the largest size of a subset of $G$ with pairwise distinct slopes is
at least $\Omega(n^{1/4})$ and at most $O(n^{2/5})$. Soon after,
Zhang~\cite{Zha93} raised 
the lower bound to $\Omega((n/\log{n})^{1/3})$, and very recently Clemen~\cite{Cle24} 
sharpened it to $\Omega((n\log\log{n}/\log{n})^{1/3})$, and these are the current records.

We consider a generalization of the distinct slope problem to arbitrary planar
point sets. Let $h(n,s)$ be the largest integer, such that any set
$P\subseteq\mathbb{R}^2$ of $n$ points with at most $s$ points collinear
contains a subset of $h(n,s)$ points that determine pairwise distinct slopes. We
prove the following bounds on $h(n,s)$ for different ranges of $s$.
In particular, when $s=\Theta(\sqrt{n})$, our lower bound on $h(n,s)$ matches the
bound by Zhang~\cite{Zha93} for the $\sqrt{n} \times \sqrt{n}$ grid. 

\begin{theorem} \label{thm:distinct-slopes}
Let $n,s\in\mathbb{N}$ with $2\leq s=O(\sqrt{n})$. We have that
\begin{equation}
\label{thm:upph1}
h(n,s)=\Omega((n/\log{s})^{1/3}).
\end{equation}
On the other hand, we have
\begin{equation}
\label{thm:upph2}
h(n,s)=
\begin{cases}
O(\sqrt{n}) & \text{for $2 \leq s \leq n^{3/8}$,}\\
\noalign{\vskip9pt}
O((n/s)^{4/5}) & \text{for $n^{3/8} \leq s =O(\sqrt{n})$.}
\end{cases}
\end{equation}
\end{theorem}
The lower bound in~\eqref{thm:upph1} is derived by selecting a random subset, with key parameters controlled
using the Szemer\'edi--Trotter theorem from incidence geometry.  
The construction proving the first upper bound in~\eqref{thm:upph2} relies on the theory of Sidon sets from
additive combinatorics. Meanwhile, the construction obtaining the second upper bound in~\eqref{thm:upph2} is a
random subset of the integer grid $[n]^2$. 

The organization of the paper is as follows. We first prove Proposition~\ref{new} in
Section~\ref{sec:general_low}. Theorem~\ref{law} is established in
Section~\ref{sec:general_up}. In Section~\ref{sec:monotone} we prove
Theorems~\ref{thm:mon-gp} and~\ref{thm:mon-gp-grid}. Section~\ref{sec:slopes} is
dedicated to proving Theorem~\ref{thm:distinct-slopes}. Finally,
Section~\ref{sec:consequences} explores additional consequences of our main results.

\section{Selecting a general position subset: lower bound}
\label{sec:general_low}

In this section we prove Proposition~\ref{new}.
The main tool behind our proof is the following result of
Cooper and Mubayi~\cite{CM16} on the chromatic number of sparse
hypergraphs. Given positive integers $j\leq k$ and a $k$-uniform hypergraph
$\mathcal{H}$, we denote the \emph{maximum $j$-degree} of $\mathcal{H}$ by  
\begin{equation}
  \Delta_j(\mathcal{H}):=\max_{A\subseteq V(\mathcal{H}), |A|=j }\left\lvert\{e\in{E(\mathcal{H})}:\,A\subseteq{e}\}\right\rvert.
\end{equation}
The \emph{maximum degree} of $\mathcal{H}$ is denoted by $\Delta(\mathcal{H}):=\Delta_1(\mathcal{H})$.
\begin{lemma}[{Cooper--Mubayi~\cite[Theorem 5]{CM16}}]
\label{dry}
Let $k\geq3$ be a fixed integer and $\mathcal{H}$ be a $k$-uniform hypergraph with maximum degree
$\Delta=\Delta(\mathcal{H})$. If 
\begin{equation}
\Delta_{j}(\mathcal{H})\leq\Delta^{\frac{k-j}{k-1}}/\gamma
\end{equation}
for $j=2,\dots,k-1$, where $\gamma=\gamma(\mathcal{H})>1$, then
\begin{equation}
\chi(\mathcal{H})=O\left(\left(\frac{\Delta}{\log{\gamma}}\right)^{\frac{1}{k-1}}\right).
\end{equation}
\end{lemma}

To apply Lemma~\ref{dry}, for any given point set $P\subseteq\mathbb{R}^2$ we define a $3$-uniform hypergraph on $P$,
whose edges are the collinear triples in $P$. The following result of Payne and Wood~\cite{PW13} gives an
upper bound on the number of edges in this auxiliary hypergraph. 
\begin{lemma}[{Payne--Wood~\cite[Lemma 2.1]{PW13}}]
\label{low}
Let $P\subseteq\mathbb{R}^{2}$ be a set of $n$ points with at most $s$ points collinear.
Then the number of collinear triples in $P$ is at most
\begin{equation}
O(n^{2}\log{s}+ns^{2}).
\end{equation}
\end{lemma}

In our proof, we want to sparsify our auxiliary hypergraph so that its maximum
degree is not too large. Meanwhile, we do not want its maximum degree to
be too small. To this end, we prove the following lemma, which allows us to
boost the maximum degree of our auxiliary hypergraph without increasing the
maximum $2$-degree too much.

\begin{lemma}
\label{eye}
Let $n,s\in\mathbb{N}$ with $n\geq{s}\geq2$. There exists a $3$-uniform hypergraph $\mathcal{H}$ on $n$ vertices,
such that
\begin{equation}
\Delta(\mathcal{H})=\Theta(n\log{s}+s^{2})\quad\text{and}\quad\Delta_{2}(\mathcal{H})=O(s).
\end{equation}
\end{lemma}
\begin{proof}[Proof of Lemma~\ref{eye}]
We may assume that $n$ is sufficiently large and $n=2m+1$. Let
\begin{equation}
q=\Theta(\log{s}+s^{2}/n)\leq m
\end{equation}
and let $S_{1},\dots,S_{q}$ be pairwise disjoint perfect matchings in a clique on the vertex set $[2m]$. Let
$\mathcal{H}$ be a $3$-uniform hypergraph on the vertex set $[n]$, whose edge set is 
\begin{equation}
\left\{\{n,x,y\}\,:\,\{x,y\}\in\bigcup_{i=1}^{q}S_{i}\right\}.
\end{equation}
Then $\Delta(\mathcal{H})$ is equal to the degree of vertex $n$, namely,
$\Delta(\mathcal{H})=mq=\Theta(n\log{s}+s^{2})$.
For any two fixed vertices, the number of edges containing them is at most the number of perfect matchings,
hence, $\Delta_{2}(\mathcal{H})\leq q = O(s)$. 
\end{proof}

\subsection{Proof of Proposition~\ref{new}}

Suppose that $n$ is sufficiently large and $s\ll\sqrt{n\log{n}}$. Let
$P\subseteq\mathbb{R}^2$ be an arbitrary set of $n$ points with at most $s$
points collinear. By Lemma~\ref{low} there are at most $O(n^{2}\log{s})$
collinear triples in $P$. Let $\tilde{P}\subseteq{P}$ be a subset of $P$
obtained by only retaining $\lvert{P}\rvert/2$ points that are contained in the
least number of collinear triples. Now every point in $\tilde{P}$ is contained
in at most $O(\lvert\tilde{P}\rvert\log{s})$ collinear triples. Let
$\mathcal{H}=\mathcal{H}(\tilde{P})$ be a $3$-uniform hypergraph on the vertex
set $\tilde{P}$, whose edges are the collinear triples in $\tilde{P}$. As there is a one-to-one correspondence
between the general position subsets of $\tilde{P}$ and the independent sets in $\mathcal{H}$, it suffices to
show that  
\begin{equation}
\alpha(\mathcal{H})=\Omega\left(\sqrt{\frac{n\log{\frac{n\log{s}}{s^2}}}{\log{s}}}\right).
\end{equation}
Recall that
\begin{equation}
\Delta(\mathcal{H})=O(\lvert\tilde{P}\rvert\log{s}).
\end{equation}
Moreover, since $\tilde{P}$ does not contain a collinear $(s+1)$-tuple, any pair of points appears in at most
$s-2$ collinear triples. This gives  
\begin{equation}
\Delta_{2}(\mathcal{H})\leq{s-2}.
\end{equation}
By Lemma~\ref{eye} there exists a $3$-uniform hypergraph $\mathcal{H}'$ on the vertex set $\tilde{P}$ satisfying
\begin{equation}
\Delta(\mathcal{H}')=\Theta(\lvert\tilde{P}\rvert\log{s})\quad\text{and}\quad\Delta_{2}(\mathcal{H}')=O(s).
\end{equation}
Let $\mathcal{H}''$ be the union of $\mathcal{H}$ and $\mathcal{H}'$, that is, a
$3$-uniform hypergraph with the vertex set $\tilde{P}$ and the edge set
$E(\mathcal{H})\cup{E(\mathcal{H}')}$. Then we have, letting
$\Delta=\Delta(\mathcal{H}'')$, 
\begin{equation}
\Delta=\Theta(\lvert\tilde{P}\rvert\log{s})\quad\text{and}\quad\Delta_{2}(\mathcal{H}'')=O(s)\leq\sqrt{\Delta}/\gamma,
\end{equation}
where
$\gamma=\Theta\left(\sqrt{\frac{\lvert\tilde{P}\rvert\log{s}}{s^2}}\right)>1$. Since
$\mathcal{H}$ is a subgraph of $\mathcal{H}''$, Lemma~\ref{dry} implies that 
\begin{equation}
\label{due}
\chi(\mathcal{H})\leq\chi(\mathcal{H}'')=O\left(\sqrt{\frac{\Delta}{\log{\gamma}}}\right)=
O\left(\sqrt{\frac{\lvert\tilde{P}\rvert\log{s}}{\log{\frac{\lvert\tilde{P}\rvert\log{s}}{s^{2}}}}}\right).  
\end{equation}
We complete the proof by noting that
$\alpha(\mathcal{H})\geq{v(\mathcal{H})/\chi(\mathcal{H})}$ and
$v(\mathcal{H})=\lvert \tilde{P}\rvert=n/2$.
\qed

\section{Selecting a general position subset: upper bounds}
\label{sec:general_up}

In this section we prove Theorem~\ref{law}, i.e.,

\begin{equation}
f(n,s)= 
\begin{cases}
O(n^{5/6+o(1)}/\sqrt{s}) & \text{for } 3 \leq s \leq n^{1/3}.
 \hfill \textbf{\quad\quad\quad part (i)}\\
\noalign{\vskip9pt}
O(n/s) & \text{for } n^{1/3} \leq s \leq n. \hfill \textbf{\quad\quad\quad part (ii)}
\end{cases}
\end{equation}
\subsection{Proof of Theorem~\ref{law} part (ii)}
Consider the 2-dimensional grid $[n]^2$. Let $\mathcal{L}$ be the set of lines
that go through at least two points in $[n]^{2}$. Note that
$\lvert\mathcal{L}\rvert\leq{n^{4}}$ and every line in $\mathcal{L}$ contains at
most $n$ points in $[n]^2$. Moreover, by the pigeonhole principle, every set of
$2n+1$ points in $[n]^2$ contains a collinear triple. Let $P$ be a random subset
of $[n]^2$ by retaining each point independently with probability
$n^{\alpha-1}/2$, where $\alpha \in [1/2,1]$. Then by Chernoff's bound, see, \eg,~\cite[Chap.~4]{MU17}, it
holds that $|P|=\Theta(n^{1+\alpha})$ with high probability. Furthermore, using the union bound and Chernoff's
bound we have  
\begin{equation}
\begin{aligned}
\label{unionbound}
\mathbb{P}\left(P\text{ contains no collinear
  $n^{\alpha}$-tuple}\right)&=1-\mathbb{P}\left(\exists\,\ell\in\mathcal{L}:\,\lvert{P\cap\ell}\rvert\geq{n^{\alpha}}\right)\\ 
&\geq1-\sum_{\ell\in\mathcal{L}}\mathbb{P}\left(\lvert{P\cap\ell}\rvert
\geq\left(1+\frac{n^{\alpha}}{2\mathbb{E}(\lvert{P\cap\ell}\rvert)}\right)\mathbb{E}(\lvert{P\cap\ell}\rvert)\right)\\
&\geq1-n^{4}\exp(-n^{\alpha}/10),
\end{aligned}
\end{equation}
which tends to $1$ as $n\to\infty$. Namely, there exists a point set $P\subseteq[n]^{2}$ of size $\Theta(n^{1+\alpha})$,
so that $P$ contains no collinear $n^{\alpha}$-tuple and every subset of $P$ of size $2n+1$ contains a collinear triple.
After scaling we obtain  
\begin{equation} 
f(n,n^{\frac{\alpha}{1+\alpha}})=O(n^{\frac{1}{1+\alpha}}) \quad \text{for}\quad \frac{1}{2} \leq \alpha \leq 1,
\end{equation}
or equivalently
\begin{equation}
f(n,s)=O(n/s) \quad \text{for}\quad n^{1/3} \leq s \leq n^{1/2}.
\end{equation}
For $n^{1/2} \leq s \leq n$, we simply consider the grid
$[s]\times[n/s]$. Obviously, this grid is a set of $n$ points with at most $s$
points on the same line. Furthermore, by the pigeonhole principle any $2n/s+1$
points contain a collinear triple. Hence, we have $f(n,s)=O(n/s)$.
\qed

\subsection{Proof of Theorem~\ref{law} part (i)}
Recall that $3 \leq s \leq n^{1/3}$.
For this proof we shall use the hypergraph container method. This method was introduced independently by Balogh, Morris,
and Samotij~\cite{BMS15}, and by Saxton and Thomason~\cite{ST15}.
Here we use a version by Saxton and Thomason. For a hypergraph $\mathcal{H}$ and $C\subseteq{V(\mathcal{H})}$,
let $\mathcal{H}[C]$ denote the subhypergraph of $\mathcal{H}$ induced by $C$.
\begin{lemma}[{Saxton--Thomason~\cite[Corollary 3.6]{ST15}}]
\label{die}
Let $\mathcal{H}$ be an $N$-vertex $k$-uniform hypergraph, and $0<\eps,\tau<1/2$. Define
\begin{equation}
\Delta(\mathcal{H},\tau):=2^{\binom{k}{2}-1}\sum_{j=2}^{k}\frac{\Delta_{j}(\mathcal{H})}{d\tau^{j-1}2^{\binom{j-1}{2}}},
\end{equation}
where $d=d(\mathcal{H})$ denotes the average degree of $\mathcal{H}$. Suppose
that $\tau<1/(1000\cdot{k}\cdot{k!^{2}})$ and
$\Delta(\mathcal{H},\tau)\leq\eps/(100k!)$. Then there exists
$c=c(k)\leq1000\cdot{k}\cdot{k!^{3}}$ and a collection
$\mathcal{C}\subseteq2^{V(\mathcal{H})}$, such that 
\begin{enumerate}
\item every independent set in $\mathcal{H}$ is a subset of some $C\in\mathcal{C}$;
\item for every $C\in\mathcal{C}$, $e(\mathcal{H}[C])\leq\eps{e(\mathcal{H})}$;
\item $\log\lvert\mathcal{C}\rvert\leq{cN\tau\log(1/\eps)\log(1/\tau)}$.
\end{enumerate}
\end{lemma}

Given $S\subseteq[n]^3$, we define $\mathcal{H}_S$ as the $3$-uniform hypergraph on the vertex set $S$, whose
edges are the collinear triples spanned by $S$. A supersaturation result by Balogh and Solymosi~\cite{BS18},
see also Keller and Smorodinsky~\cite{KS21}, shows that $e(\mathcal{H}_S)$ is large whenever $\lvert
S\rvert\geq n^{2+c}$ for any fixed $c>0$. We further prove that in this case not only $\mathcal{H}_S$ contains
many edges, but also there is a spanning subhypergraph $\mathcal{H}'\subseteq\mathcal{H}_S$, i.e., a
subhypergraph on the same vertex set $S$, such that $\mathcal{H}'$ has almost as many edges as $\mathcal{H}_S$
and the co-degrees of $\mathcal{H}'$ are well-bounded. 
\begin{lemma}
\label{arm}
Fix any $0<f<1$. Let $n\in\mathbb{N}$ be sufficiently large and $0\leq x \leq 1-f$. For every set
$S\subseteq[n]^{3}$ with $\lvert{S}\rvert=n^{3-x}$, there exists a spanning subhypergraph
$\mathcal{H}'\subseteq\mathcal{H}_S$ such that  
\begin{equation}
e(\mathcal{H}')\geq\frac{n^{6-4x}}{10^7\log{n}}, \quad \quad \Delta(\mathcal{H}') \leq \frac{n^{3-3x}}{10^3f\log{n}},
\quad \quad \Delta_2(\mathcal{H}')\leq 1. 
\end{equation}
\end{lemma}
\begin{proof}[Proof of Lemma~\ref{arm}]
Set $t=100n^{x}$ and let
\begin{equation}
U:=\left\{(a,b,c)\in\mathbb{Z}^{3}:\,1\leq{a}\leq{2n/t},\,-n\leq{b,c}\leq{n}\right\}
\end{equation}
and
\begin{equation}
V:=\left\{(a,b,c)\in\mathbb{N}^{3}:\,n/t\leq{a}\leq2n/t,\,a>\max\{b,c\},\,a\text{ is a prime}\right\}.
\end{equation}
We define a family $\mathcal{L}=\mathcal{L}(t)$ of lines, whose starting points are in $U$ and whose
directions are in $V$. We regard $\mathcal{L}$ as a family of geometric lines, so repeated descriptions of the
same line are identified. Note that these lines are defined in $\mathbb{R}^{3}$ and we are interested in their
intersection with $S$. It is easy to see that every line in $\mathcal{L}$ passes through at most $t$ points of
$[n]^{3}$. Since the number of primes smaller than $m$ is at least $m/\log{m}$ and at most $1.1m/\log{m}$ when
$m$ is sufficiently large (see~\cite[Theorem 1]{RS62}), we have that 
\begin{equation}
\label{eq:|V|}
\frac{n^3}{t^3\log(n/t)}\leq\lvert V\rvert\leq\frac{100n^3}{t^3\log(n/t)}.
\end{equation}

Moreover, the family $\mathcal{L}$ has the following properties.
\begin{claim}
\label{say}
For each $(a,b,c)\in{V}$,
\begin{enumerate}
\item the lines in $\mathcal{L}$ with direction $(a,b,c)$ cover the points in $[n]^{3}$;
\item the number of lines in $\mathcal{L}$ with direction $(a,b,c)$ is at most $10n^{3}/t$;
\item every point in $[n]^3$ is contained in at most $\frac{100n^3}{t^3\log(n/t)}$ lines from $\mathcal{L}$.
\end{enumerate}
\end{claim}
\begin{proof}[Proof of Claim~\ref{say}]
Fix an arbitrary $(a,b,c)\in V$. For any $(p_1,p_2,p_3)\in[n]^3$, let $q=\lfloor\frac{p_1-1}{a}\rfloor$. Then
\begin{equation}
1\leq p_1-qa\leq a\leq\frac{2n}{t}.
\end{equation}
Moreover, since $1\leq b,c<a$ and $qa\le p_1-1\le n-1$, we have
\begin{equation}
2-n\leq p_2-qb\leq n\quad\text{and}\quad2-n\leq p_3-qc\leq n.
\end{equation}
Thus, the point $(p_1,p_2,p_3)$ is contained in the line in $\mathcal{L}$ with direction $(a,b,c)$ that goes
through $(p_1-qa,p_2-qb,p_3-qc)\in U$. This proves item (1). 

For item (2), since each point of $U$ determines exactly one line with direction $(a,b,c)$, the number of
lines in $\mathcal{L}$ with direction $(a,b,c)$ is at most 
\begin{equation}
|U|\leq\frac{2n}{t}\cdot(2n+1)^2\leq\frac{10n^3}{t}.
\end{equation}

Finally, for each $v\in V$ there is at most one line with direction $v$ through any fixed point of $[n]^3$.
Therefore, by~\eqref{eq:|V|} every point of $[n]^3$ is contained in at most
\begin{equation}
|V|\leq\frac{100n^3}{t^3\log(n/t)}
\end{equation}
lines from $\mathcal{L}$, proving item (3).
\end{proof}
We now construct the spanning subhypergraph $\mathcal{H}'\subseteq \mathcal{H}_S$ as follows. For each line
$\ell\in \mathcal{L}$ we place an arbitrary matching (each matching edge is a collinear triple) on the
vertices $\ell\cap S$ covering all but at most $2$ vertices. The edge set of $\mathcal{H}'$ is the set of all
the matching edges. Then it holds that $\Delta_2(\mathcal{H}')\leq 1$ and by Claim~\ref{say} item (3), 
\begin{equation}
\Delta(\mathcal{H}')\leq \frac{100n^3}{t^3\log(n/t)}\leq \frac{n^{3-3x}}{10^3f\log{n}}.
\end{equation}

Let $\mathcal{L}_v\subseteq\mathcal{L}$ be the family of lines with direction $v\in{V}$. By Claim~\ref{say}
item (2), $|\mathcal{L}_v|\leq10n^{3}/t$. Together with Claim~\ref{say} item (1) we have that the number $T_v$
of edges in $\mathcal{H}'$ coming from triples on lines in $\mathcal{L}_v$ is at least 
\begin{equation}
  T_v=\sum_{\ell\in\mathcal{L}_v} e(\mathcal{H}'[\ell\cap S])=
  \sum_{\ell\in\mathcal{L}_v} \left\lfloor \frac{\lvert\ell\cap{S}\rvert}{3} \right\rfloor \geq 
\frac{|S|}{3}-\lvert\mathcal{L}_v\rvert\geq\frac{n^{3-x}}{5}.
\end{equation}
Since the directions in $V$ are pairwise non-parallel, we have
$\mathcal{L}_v\cap\mathcal{L}_{v'}=\emptyset$ for any distinct $v,v'\in V$. Then, since
$\lvert{V}\rvert\geq\frac{n^3}{t^3\log(n/t)}$ and $t=100n^x$, we obtain
\begin{equation}
e(\mathcal{H}')=\sum_{v\in V}T_{v} \geq \frac{n^3}{t^3\log(n/t)} \cdot \frac{n^{3-x}}{5}\geq\frac{n^{6-4x}}{10^7\log{n}}.
\end{equation}
This completes the proof of Lemma~\ref{arm}.
\end{proof}

Our proof of Theorem~\ref{law} part (i) extends the framework of that by Balogh and Solymosi~\cite{BS18},
which consists of three steps. We start with the $3$-dimensional grid
$[n]^3$. In the first step, we iteratively apply the hypergraph container lemma
on an auxiliary hypergraph until we obtain a collection
$\mathcal{C}\subseteq{2^{[n]^3}}$ of small containers, such that every subset of
$[n]^3$ without collinear triples is a subset of one of these containers. In the
second step, we take a random subset of $[n]^3$ and show that with high
probability this set contains no collinear $s$-tuple and also does not contain
any large subset without collinear triples. Note that the collection
$\mathcal{C}$ from the first step will be used here to conquer the union
bound. In the last step, we project this set from $[n]^3$ into the plane and
obtain an upper bound construction. The main novelty in our proof is that in each iteration we do not apply
the hypergraph container lemma on $\mathcal{H}_S$ with $S\subseteq[n]^3$ but on a sparsened auxiliary
hypergraph $\mathcal{H}'\subseteq\mathcal{H}_S$ given by Lemma~\ref{arm}. The merit of this step is that the
co-degrees of $\mathcal{H}'$ are well-bounded, which enables us to get a better collection of containers and
ultimately obtain better upper bounds.

\begin{proof}[Proof of Theorem~\ref{law} part (i)]
Note that for $3 \leq s\leq \log{n}$, by the monotonicity of the function $f(n,s)$ we have that
\begin{equation}
f(n,s)\leq f(n,3) \leq n^{5/6+o(1)} = O(n^{5/6+o(1)}/\sqrt{s}).
\end{equation}
Hence, it suffices to prove the claimed upper bounds for
$\log{n}\leq s\leq n^{1/3}$. Fix any $0<f<1/100$ and assume that $n\in\mathbb{N}$ is
sufficiently large. Let
\begin{equation}
\alpha\in \left[\frac{2\log\log{n}}{\log{n}},1\right] \quad \text{and} \quad \gamma=\frac{2 \alpha+1}{3}\leq 1.
\end{equation}

For every subset $S\subseteq[n]^{3}$ of size $n^{3-x}$ with
$0\leq{x}\leq\gamma-f$, we can do the following. Let
$\mathcal{H}_S$ be the $3$-uniform hypergraph on the vertex set $S$,
whose edges are the collinear triples spanned by
$S$. Let $\mathcal{H}'\subseteq\mathcal{H}_S$ be the spanning subhypergraph guaranteed by
Lemma~\ref{arm}. Since $e(\mathcal{H}')\geq\frac{n^{6-4x}}{10^7\log{n}}$, the average degree
$d=d(\mathcal{H}')$ of $\mathcal{H}'$ is at least 
$\frac{n^{3-3x}}{10^{7}\log{n}}$. Moreover, we have that 
\begin{equation}
\Delta_{2}(\mathcal{H}')\leq1 \quad \text{and} \quad \Delta_{3}(\mathcal{H}')\leq1.
\end{equation}
Let $\tau=n^{x+\gamma/2-3/2}$. Since $x\leq\gamma-f$ and $\gamma\leq1$, it holds that $\tau\leq{n^{-f}}\ll1$.
Furthermore, 
\begin{equation}
\begin{aligned}
\Delta(\mathcal{H}',\tau)&=\frac{4\Delta_{2}(\mathcal{H}')}{d\tau}+\frac{2\Delta_{3}(\mathcal{H}')}{d\tau^{2}}\leq
\frac{4}{\frac{n^{3-3x}}{10^{7}\log{n}}\cdot
  n^{x+\frac{\gamma}{2}-\frac{3}{2}}}  +\frac{2}{\frac{n^{3-3x}}{10^{7}\log{n}}
  \cdot n^{2x+\gamma-3}}  \\ 
&\leq 
\frac{10^{8}\log{n}}{n^{2f}}+\frac{10^{8}\log{n}}{n^{f}}\ll\eps,
\end{aligned}
\end{equation}
where $\eps=n^{-f/2}<1/2$. Then by Lemma~\ref{die} there exists a collection $\mathcal{C}\subseteq2^{S}$, such that
\begin{enumerate}
\item every independent set in $\mathcal{H}'$ is a subset of some $C\in\mathcal{C}$;
\item for every $C\in\mathcal{C}$, $e(\mathcal{H}'[C])\leq{n^{-f/2}e(\mathcal{H}')}$;
\item $\lvert\mathcal{C}\rvert\leq\exp\left(10^{7}n^{(3+\gamma)/2}(\log{n})^{2}\right)$.
\end{enumerate}
Since $\mathcal{H}'\subseteq \mathcal{H}_S$, every independent set in $\mathcal{H}_S$
is a subset of some $C\in \mathcal{C}$. By item (1) we have for each container $C\in \mathcal{C}$ that
\begin{equation}
\label{shrink}
e(\mathcal{H}')-\lvert S\backslash C\rvert \cdot \Delta(\mathcal{H}') \leq e(\mathcal{H}'[C]) \leq n^{-f/2}e(\mathcal{H}').
\end{equation}
Recall that $\lvert S\rvert=n^{3-x}$ and by Lemma~\ref{arm}, $e(\mathcal{H}')\geq\frac{n^{6-4x}}{10^7\log{n}}$
and $\Delta(\mathcal{H}')\leq \frac{n^{3-3x}}{10^3f\log{n}}$. After rearranging~\eqref{shrink} we get
\begin{equation}
  \lvert S\backslash C\rvert\geq  \frac{e(\mathcal{H}')}{\Delta(\mathcal{H}')} (1-n^{-f/2})
  \geq \frac{fn^{3-x}}{10^5}=\frac{f}{10^5}\cdot\lvert{S}\rvert,
\end{equation}
which implies that $\lvert{C}\rvert\leq(1-10^{-5}f)\lvert{S}\rvert$.

We apply the above procedure first on $S_{0}=[n]^{3}$ and obtain a collection of
containers. Then, we iterate the procedure on the containers until every
container has size at most $n^{3-\gamma+f}$. This process takes at most $(\log{n})^2$ iterations. Indeed, in
each iteration we shrink the container size by a $(1-10^{-5}f)$ factor. Since $f$ is fixed and $n$ is
sufficiently large, we have $(1-10^{-5}f)^{(\log{n})^2}<1/n$. Namely, after at most $(\log{n})^2$ iterations
the size of every container is at most $n^{3-\gamma+f}$.

Note that any subset of $[n]^{3}$ with no collinear triple is an independent set
in $\mathcal{H}_{S_{0}}$ and thus is contained in one of the final
containers. Overall, we obtain a collection $\mathcal{C}$ of 
\begin{equation}
\exp\left(10^{7}n^{(3+\gamma)/2}(\log{n})^{4}\right)\leq\exp\left(n^{(3+\gamma)/2+f}\right)
\end{equation}
containers, such that every subset of $[n]^{3}$ without collinear triples is
contained in one of these containers. Moreover, each container has size at most
$n^{3-\gamma+f}$.

Let $P$ be a random subset of $[n]^3$ by retaining each grid point independently with probability
\begin{equation}
p=\frac{1}{2}n^{\alpha-1}=\frac{1}{2}n^{(3\gamma-3)/2}.
\end{equation}
Then by Chernoff's bound we have that $|P|=\Theta(n^{2+\alpha})$ with high
probability. Applying the union bound and Chernoff's bound as in~\eqref{unionbound} we also have that $P$
contains no collinear $n^\alpha$-tuple 
with high probability. Furthermore, we say that a subset of $P$ is \emph{bad} if it has size
$n^{(3+\gamma)/2+2f}$ and spans no collinear triple. When $\gamma>1-2f/3$, since each container has size at
most $n^{3-\gamma+f}<n^{(3+\gamma)/2+2f}$, $P$ contains no bad subset. When $0\leq \gamma\leq1-2f/3$, the
expected number of bad subsets of $P$ is at most  
\begin{equation}
\begin{aligned}
  \sum_{C\in\mathcal{C}}\binom{\lvert{C}\rvert}{n^{(3+\gamma)/2+2f}}p^{n^{(3+\gamma)/2+2f}}\leq
  \exp\left(n^{(3+\gamma)/2+f}\right)\binom{n^{3-\gamma+f}}{n^{(3+\gamma)/2+2f}}\left(n^{(3\gamma-3)/2}\right)^{n^{(3+\gamma)/2+2f}} \ll1.
\end{aligned}
\end{equation}
By the first moment method, we have that with high probability $P$ does not contain any bad subset.
Consequently, there exists a set $P\subseteq[n]^{3}$ of size $\Theta(n^{2+\alpha})$, so
that $P$ contains no collinear $n^{\alpha}$-tuple and every subset of $P$ of
size $n^{(3+\gamma)/2+2f}$ spans a collinear triple. We can project $P$ into $\mathbb{R}^{2}$ in such a way
that all the collinear relations remain the same. After scaling we obtain 
\begin{equation} 
f(n,n^{\frac{\alpha}{2+\alpha}})=O\left(n^{\frac{(3+\gamma)/2+2f}{2+\alpha}}\right)=O\left(n^{\frac{5+\alpha+6f}{6+3\alpha}}\right)
\quad \text{for}\quad \alpha
\in\left[\frac{2\log\log{n}}{\log{n}},1\right], 
\end{equation}
which implies
\begin{equation}
f(n,n^{\beta})=O\left(n^{\frac{5-3\beta+6f}{6}}\right) \quad \text{for}\quad \beta \in
\left[\frac{\log\log{n}}{\log{n}},\frac{1}{3}\right]. 
\end{equation}
Further, letting $s=n^{\beta}$,
\begin{equation}
\label{mew}
f(n,s)=O\left(n^{\frac{5+6f}{6}}/\sqrt{s}\right) \quad \text{for}\quad \log{n} \leq s \leq n^{1/3}.
\end{equation}
Since $f$ can be taken arbitrarily small in the beginning, we obtain the claimed upper bound.
\end{proof}

\section{Selecting a monotone general position subset}

In this section we prove Theorems~\ref{thm:mon-gp} and~\ref{thm:mon-gp-grid}. \label{sec:monotone}

\subsection{Proof of Theorem~\ref{thm:mon-gp}}

\emph{Lower bound.} Let $P\subseteq\mathbb{R}^2$ be an arbitrary set of $n$
points without collinear $(s+1)$-tuples. First, we have a general position subset of $P$ of size $f(n,s)$. Then, by the
Erd\H{o}s--Szekeres theorem we can extract a monotone subset from it, of size at least $\sqrt{f(n,s)}$. Recall
the lower bounds on $f(n,s)$ from~\eqref{pay} and Proposition~\ref{new}, we obtain 
\begin{equation}
g(n,s)=
\begin{cases}
\Omega\left(\left(\frac{n\log{\frac{n\log{s}}{s^2}}}{\log{s}}\right)^{1/4}\right) & \text{for $3\leq s\ll\sqrt{n\log{n}}$,}\\
\noalign{\vskip9pt}
\Omega\left(\sqrt{n/s}\right) & \text{for $s=\Omega(\sqrt{n\log{n}})$.}
\end{cases}
\end{equation}

\medskip
\emph{Upper bound.} Let $k= \sqrt{n/s}$, and assume without loss of generality that $k\in\mathbb{N}$. Consider
a $k \times k$ grid and replace each grid point by a circle of diameter $\epsilon>0$, where $\epsilon$ is
chosen to be sufficiently small. Slightly perturb these circles so that no line intersects with three of
them. Then, we place a collinear $s$-tuple spanning a line segment of length $\eps$ inside every circle and
slightly rotate these line segments so that the extension of each line segment does not intersect with other
line segments. See Figure~\ref{fig:grid2} for an illustration. 
\begin{figure}[htbp]
\centering
\begin{tikzpicture}[scale=0.7]

  % Draw four equally spaced collinear points.
  % Arguments:
  %   #1 = x-coordinate of center
  %   #2 = y-coordinate of center
  %   #3 = rotation angle
  \newcommand{\cluster}[3]{%
    \begin{scope}[shift={(#1,#2)},rotate=#3]
      \foreach \t in {-0.27,-0.09,0.09,0.27}{
        \fill (\t,0) circle (1.25pt);
      }
    \end{scope}%
  }

  % Top row
  \cluster{0}{4.0}{-28}
  \cluster{3.1}{4.0}{82}
  \cluster{6.2}{4.0}{-72}

  % Middle row
  \cluster{0}{2.0}{0}
  \cluster{3.1}{2.0}{-15}
  \cluster{6.2}{2.0}{-32}

  % Bottom row
  \cluster{0}{0}{-68}
  \cluster{3.1}{0}{68}
  \cluster{6.2}{0}{-10}

\end{tikzpicture}
\caption{Upper bound construction for Theorem~\ref{thm:mon-gp} when $n=36$ and $s=4$.}
\label{fig:grid2}
\end{figure}
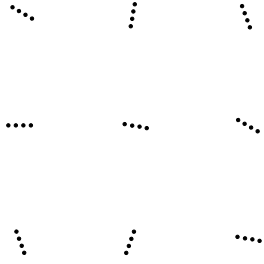

In this construction there are $n$ points with at most $s$ collinear points. 
Let $P$ denote the resulting set and let $A \subseteq P$ be a monotone general position subset. Note that
\begin{enumerate}[(i)]
\item $A$ contains at most two points from each collinear $s$-tuple;
\item the largest monotone subset of the $k \times k$ grid that we started with has at most $2k-1$ points.
\end{enumerate}
Thus, $|A| \leq 4k-2 = O(\sqrt{n/s})$ as claimed.
\qed

\subsection{Proof of Theorem~\ref{thm:mon-gp-grid}}

\emph{Lower bound.}
According to a classical result of Jarn\'ik~\cite{Jar1926},
a strictly convex curve of length $L$ in the plane can be incident to at most
$3 (2\pi)^{-1/3} L^{2/3} (1 + o(1))$ lattice points and this estimate is the best possible.
In particular, Jarn\'ik's polygon in the $\sqrt{n} \times \sqrt{n}$ grid matches this bound;
see also~\cite{BP06}. One can then select a monotone arc of this polygon with
$\Omega(n^{1/3})$ vertices. Since the polygon is convex, the arc contains no collinear triple.
We next derive a better lower bound, namely $\Omega\left((n/\log{n})^{2/5}\right)$.

Let $n\in\mathbb{N}$ be sufficiently large and assume without loss of generality
that $\sqrt{n}=2m+1$ with $m\in\mathbb{N}$. Let $x=x(n)\geq1$ with
$x\ll\sqrt{n}$ to be determined later. Let $U$ be an axis-aligned square of side length $2m$
containing a $\sqrt{n}\times\sqrt{n}$ grid $G$.  
Let $o$ be the center point of the square $U$ and let $A$ be the annulus formed by two concentric circles
of radii $m$ and $m-x$, centered at $o$. 

Gauss' circle problem asks to determine the number of points of the integer
lattice $\mathbb{Z}^2$ contained in a circle of radius $R$ around the origin.
A result by Gauss says that this number is $\pi R^2+O(R)$, see, \eg,~\cite{Ha15}. Thus, 
$\lvert{A \cap G}\rvert=  \Theta(x\sqrt{n})$. Next we show some important properties of the annulus~$A$.

\begin{figure}[ht!]

\begin{subfigure}{.49\textwidth}
	\centering
	\begin{tikzpicture}[scale=0.5]
		\foreach \i in {0,1,...,10} 
            \foreach \j in {0,1,...,10} 
                \node [vtx] (v\i\j) at (\i,\j) {};;
		
        \draw (5,5) circle [radius=5]; 
        \draw (5,5) circle [radius=3]; 
        \draw (5.5,-2) -- (7.5,12);
        \draw (7,-2) -- (9,12);
        \node [below] at (7,12) {$\ell$};
        \node [below] at (9.5,12) {$\ell'$};
        \node[below] at (v55) {$o$};
        \node [below] at (8.4,4.9) {$a$};
        \node [below] at (9,9.15) {$b$};
        \draw[gray] (5,5) -- (8.5,8.67);
        \draw[gray] (5,5) -- (7.95,4.56);
        \draw[line width=0.4mm, <->] (0.2,3.6) -- (2.09,4.02);
        \draw[line width=0.4mm, <->] (2,5) -- (5,5);
        \node [draw, shape = circle, fill = black, inner sep=1.5pt] at (5,5){};
        \node [draw, shape = circle, fill = black, inner sep=1.5pt] at (8.5,8.67){};
        \node [draw, shape = circle, fill = black, inner sep=1.5pt] at (7.95,4.56){};
        \node [below] at (1.2,3.87) {$x$};
        \node[above] at (3.53,4.93) {$m-x$};
\end{tikzpicture}
\caption*{(a)}
\end{subfigure}
\begin{subfigure}{.49\textwidth}
	\centering
	\begin{tikzpicture}[scale=0.5]
		\foreach \i in {0,1,...,10} 
            \foreach \j in {0,1,...,10} 
                \node [vtx] (v\i\j) at (\i,\j) {};;
		
        \draw (5,5) circle [radius=5]; 
        \draw (5,5) circle [radius=3]; 
        \draw (5.5,-2) -- (7.5,12);
        \node [below] at (7,12) {$\ell$};
        \node [below] at (7,4.9) {$a$};
                \draw[gray] (5,5) -- (6.5,4.7);
                \node [draw, shape = circle, fill = black, inner sep=1.5pt] at (6.5,4.7){};
        \node [below] at (7.35,7.95) {$b$};
                \draw[gray] (5,5) -- (6.85,7.38);
                \node [draw, shape = circle, fill = black, inner sep=1.5pt] at (6.85,7.38){};
        \node [below] at (7.6,10) {$c$};
                \draw[gray] (5,5) -- (7.14,9.53);
                \node [draw, shape = circle, fill = black, inner sep=1.5pt] at (7.14,9.53){};
		\node [draw, shape = circle, fill = black, inner sep=1.5pt] at (5,5){};
		\node [below] at (v55) {o};

	\end{tikzpicture}
    \caption*{(b)}
\end{subfigure}
	\caption{Part (a) illustrates the two parallel lines $\ell,\ell'$ from the proof of Claim~\ref{claim:tangent}.\\
    Part (b) illustrates the points $a,b,c$ on the line $\ell$ from the proof of Claim~\ref{claim:length}.}
	\label{fig:annulus1}
\end{figure}
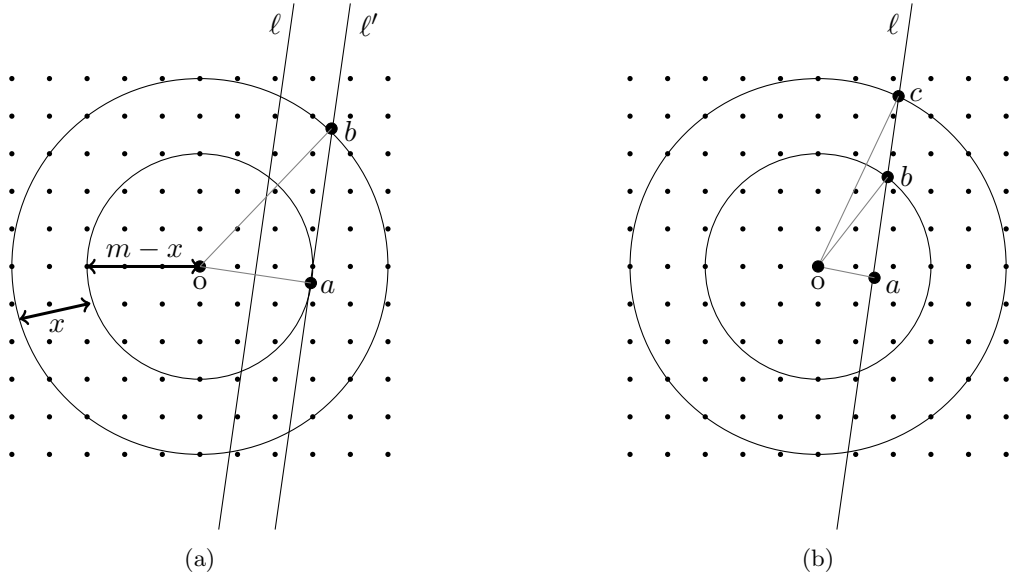

\begin{lemma}\label{triples}
The number of collinear triples in $A\cap{G}$ is $O(x^{5/2}n^{3/4}\log{n})$.
\end{lemma}

Before proving Lemma~\ref{triples}, we note that its bound is best possible up to a constant factor. More
precisely, for $x\ll\sqrt{n}$, the number of collinear triples in $A\cap{G}$ is
$\Theta(x^{5/2}n^{3/4}\log{n})$. A proof of the matching lower bound is included in the appendix.

\begin{proof}[Proof of Lemma~\ref{triples}]
Fix any slope $(q,r)\in\mathbb{Z}^{2}$ with $q$ and $r$ being coprime. Let $\mathcal{L}(q,r)$ be the family of
lines with slope $(q,r)$ that are incident to at least two points of $A\cap{G}$, and $T(q,r)$ be the number of
collinear triples in $A\cap G$ coming from lines $\ell\in\mathcal{L}(q,r)$. Due to the symmetry of $A$ we may
assume that $0\leq q\leq r$. For a set $I$ of pairwise disjoint intervals on a line, let $\mu(I)$ denote their
total length. Note that each line $\ell\in\mathcal{L}(q,r)$ contains at most $O(\mu(\ell \cap A)/r)$ points
in $A\cap G$. Hence, we first bound $\mu(\ell \cap A)$ from above.

\begin{claim}\label{claim:tangent}
  Let $\ell$ be an arbitrary line.
  Then $\mu(\ell \cap A) \leq 2 x^{1/2} n^{1/4}$.
\end{claim}
\begin{proof}[Proof of Claim~\ref{claim:tangent}]
Let $\ell'$ be the line which is parallel to $\ell$ and a tangent line to the
inner circle of $A$. Since $\mu(\ell \cap A) \leq \mu(\ell' \cap A)$, it remains
to show that $\mu(\ell' \cap A) \leq 2 x^{1/2} n^{1/4}$. Let $a$ denote the
tangent point and let $b$ denote one of the intersection points of $\ell'$ with
the outer circle. See Figure~\ref{fig:annulus1} (a) for an illustration.  

Recall that the radii of the inner circle and the outer circle are $m-x$ and
$m$, namely, $\mu(oa)=m-x$ and $\mu(ob)=m$. A straightforward trigonometric
calculation shows that 
\begin{equation}
\mu(\ell' \cap A) = 2 \mu(ab) = 2 \sqrt{\mu(ob)^2-\mu(oa)^2}\leq 2\sqrt{2mx} \leq 2x^{1/2}n^{1/4},
\end{equation}
as desired.
\end{proof}

\begin{claim}
\label{claim:length}
Let $t=t(n)\geq6$ with $tx\leq m$. If the distance $\mathrm{dist}(\ell,o)$
between a line $\ell \in \mathcal{L}(q,r)$ and the center point $o$ is $m-tx$,
then $\mu(\ell \cap A)\leq 4\sqrt{mx/t}$. 
\end{claim}
\begin{proof}[Proof of Claim~\ref{claim:length}]
Let $a\in\ell$ be the point such that the segment $oa$ is orthogonal to
$\ell$. Since $t>1$, the line $\ell$ intersects the annulus $A$ in two
segments. Let $b$ and $c$ be the intersection points of $\ell$ with the inner
and outer circles on one of the sides. See Figure~\ref{fig:annulus1} (b) for an
illustration. It follows that 
\begin{equation}
\mu(oa)= m-tx,\quad \mu(ob)=m-x,\quad \mu(oc)=m
\end{equation}
and by symmetry that $\mu(\ell \cap A)=2\mu(bc)$. According to the Pythagorean formula we have
\begin{equation}
\mu(ab)^2=\mu(ob)^2-\mu(oa)^2=2mtx-t^2x^2-2mx+x^2, \quad \mu(ac)^2=\mu(oc)^2-\mu(oa)^2=2mtx-t^2x^2.
\end{equation}
Since $\mu(ab)^2+2\mu(ab)\mu(bc)\leq(\mu(ab)+\mu(bc))^2=\mu(ac)^2$, it holds that
\begin{equation}
\mu(bc)\leq
\frac{\mu(ac)^2-\mu(ab)^2}{2\mu(ab)}=\frac{2mx-x^2}{2\sqrt{2mtx-t^2x^2-2mx+x^2}}\leq
\frac{mx}{\sqrt{2mtx-t^2x^2-2mx}}\leq2\sqrt{\frac{mx}{t}}. 
\end{equation}
Thus, $\mu(\ell\cap A)\rvert\leq 2\mu(bc) \leq 4\sqrt{mx/t}$.
\end{proof}

Next we want to bound the number of lines $\ell\in\mathcal{L}(q,r)$ with given $\mathrm{dist}(\ell,o)$.

\begin{claim}\label{claim:number}
Let $0\leq t \leq m/x$ and $\eps>0$. Then the number of lines
$\ell\in\mathcal{L}(q,r)$ with $\mathrm{dist}(\ell,o)\in[m-(t+\eps)x,m-tx]$
is at most $O(\eps xr)$. 
\end{claim}
\begin{proof}[Proof of Claim~\ref{claim:number}]
It suffices to prove that the distance between any two distinct lines in
$\mathcal{L}(q,r)$ is $\Omega(1/r)$, whereupon a simple calculation
shows that the number of lines $\ell\in\mathcal{L}(q,r)$ with
$\mathrm{dist}(\ell,o)\in[m-(t+\eps)x,m-tx]$ is $O(\eps xr)$. 

Let $\ell_1,\ell_2\in\mathcal{L}(q,r)$ be two distinct lines and let
$\mathrm{dist}(\ell_1,\ell_2)$ denote the distance between $\ell_1$ and
$\ell_2$. Since $\ell_1$ and $\ell_2$ both have the same slope $(q,r)$ and each
passes through at least two points of $A\cap G$ by definition, we can write them
as 
\begin{equation}
\ell_1:\,rx-qy+k_1=0 \quad\text{and}\quad \ell_2:\,rx-qy+k_2=0,
\end{equation}
where $k_1,k_2\in\mathbb{Z}$ and $k_1\neq k_2$. Then the formula of the distance between two parallel lines in
the plane gives that 
\begin{equation}
\label{eq:lines_distance}
\mathrm{dist}(\ell_1,\ell_2)=\frac{\lvert{k_1-k_2}\rvert}{\sqrt{r^2+q^2}}\geq
\frac{1}{\sqrt{r^2+q^2}}=\Omega\left(\frac{1}{r}\right),
\end{equation}
where in the last step we used that $0\leq q\leq r$.
\end{proof}

Now we write $T(q,r)$ as follows and compute the two parts separately.
\begin{equation}
\label{eq:lines1}
\begin{aligned}
T(q,r)&=\sum_{\ell\in\mathcal{L}(q,r)}\binom{|\ell \cap A\cap G|}{3} \\
&\leq
\sum_{\substack{\ell\in\mathcal{L}(q,r)\\\mathrm{dist}(\ell,o)\in(m-6x,m]}}O\left(\frac{\mu(\ell
    \cap A)^3}{r^3}\right) +
  \sum_{\substack{\ell\in\mathcal{L}(q,r)\\ \mathrm{dist}(\ell,o)\in[0,m-6x]}}O\left(\frac{\mu(\ell
    \cap A)^3}{r^3}\right), 
\end{aligned}
\end{equation}
where we used in the last step that $|\ell \cap A\cap G|=O(\mu(\ell \cap A)/r)$.
By Claim~\ref{claim:number} the number of lines $\ell\in\mathcal{L}(q,r)$ with
$\mathrm{dist}(\ell,o)\in[m-6x,m]$ is at most $O(xr)$ and by
Claim~\ref{claim:tangent} $\mu(\ell \cap A) \leq 2 x^{1/2} n^{1/4}$ holds for all
$\ell\in\mathcal{L}(q,r)$. Hence, we have that 
\begin{equation}
\label{eq:lines2}
\sum_{\substack{\ell\in\mathcal{L}(q,r)\\\mathrm{dist}(\ell,o)\in(m-6x,m]}}O\left(\frac{\mu(\ell
    \cap A)^3}{r^3}\right)\leq O(xr) \cdot
  O\left(\frac{x^{3/2}n^{3/4}}{r^3}\right)=O\left(\frac{x^{5/2}n^{3/4}}{r^2}\right). 
\end{equation}
For any $0\leq t \leq m/x$ and $\eps>0$, by Claim~\ref{claim:number} the
number of lines $\ell\in\mathcal{L}(q,r)$ with
$\mathrm{dist}(\ell,o)\in[m-(t+\eps)x,m-tx]$ is at most $O(\eps xr)$ and
we have that $\mu(\ell \cap A)\leq4\sqrt{mx/t}$ for every such line $\ell$ by
Claim~\ref{claim:length}. Therefore, 
\begin{equation}
\label{eq:lines3}
\begin{aligned}
\sum_{\substack{\ell\in\mathcal{L}(q,r)\\ \mathrm{dist}(\ell,o)\in[0,m-6x]}}O\left(\frac{\mu(\ell
  \cap
  A)^3}{r^3}\right)&\leq\lim_{\eps\to0}\sum_{t\in\{6,6+\eps,6+2\eps,\dots,m/x\}}\sum_{\substack{\ell\in\mathcal{L}(q,r)\\
    \mathrm{dist}(\ell,o)\in[m-(t+\eps)x,m-tx]}}O\left(\frac{\mu(\ell \cap A)^3}{r^3}\right)\\ 
&\leq\lim_{\eps\to0}\sum_{t\in\{6,6+\eps,6+2\eps,\dots,m/x\}}O(\eps xr)\cdot O\left(\frac{m^{3/2}x^{3/2}}{t^{3/2}r^3}\right)\\
&=\int_{t=6}^{m/x}O\left(\frac{m^{3/2}x^{5/2}}{t^{3/2}r^2}\right)\mathrm{d}t = O\left(\frac{m^{3/2}x^{5/2}}{r^2}\right) \cdot
\int_{t=6}^{m/x} t^{-3/2}  \mathrm{d}t \\
&=O\left(\frac{x^{5/2}n^{3/4}}{r^2}\right).
\end{aligned}
\end{equation}

Note that in the calculation of the integral in~\eqref{eq:lines3} we used that $x\ll m \approx \sqrt{n}/2$ and
thus the last integral in the penultimate line of~\eqref{eq:lines3} is $O(1)$. By
combining~\eqref{eq:lines1},~\eqref{eq:lines2} and~\eqref{eq:lines3} we obtain   
\begin{equation}
\begin{aligned}
T(q,r)&=O\left(\frac{x^{5/2}n^{3/4}}{r^2}\right).
\end{aligned}
\end{equation}

Now let
\begin{equation}
S=\left\{(q,r)\in\{-2m,\dots,2m\}\times\{-2m,\dots,2m\}\text{, where }q\text{ and }r\text{ are coprime}\right\}
\end{equation}
be the set of all possible slopes of lines that go through at least two points of $A\cap{G}$.
The number of collinear triples in $A\cap G$ is 
\begin{equation}
  \sum_{(q,r)\in S}T(q,r) \leq 2 \sum_{\substack{(q,r)\in S,\\ \,0\leq q\leq r}}T(q,r)
  =O\left(\sum_{r=1}^{\sqrt{n}}r\cdot\frac{x^{5/2}n^{3/4}}{r^{2}}\right)
  =O(x^{5/2}n^{3/4}\log{n}),
\end{equation}
completing the proof of Lemma~\ref{triples}.
\end{proof}

Recall that our goal is to select a large general position subset $P$ of $A\cap{G}$, such that for any
$u,v\in{P}$ with $u_x<v_x$ we have $u_y \leq v_y$. To achieve this, one needs to avoid two types of obstacles,
collinear triples and descending pairs. A \emph{descending pair} is a pair $(u,v)$ of points with $u_x < v_x$
and $u_y > v_y$. In Lemma~\ref{triples} we have bounded 
the number of collinear triples in $A \cap G$. Next we want to bound the number
of descending pairs in a certain section of $A\cap G$. 

Let $v_1$ be the lower right vertex of the square $U$. Moreover, let $v_2$ and
$v_3$ be two points on the right and bottom sides of the square $U$,
respectively, with $\angle{v_2ov_3}=30^\circ$ and $ov_1$ being the bisector. Let
$Q$ be the quadrilateral formed by $ov_1v_2v_3$ and let $B = A \cap Q$ be a
sector of the annulus $A$. See Figure~\ref{fig:annulus4} for an
illustration. Note that $\lvert{B \cap G}\rvert=\Theta(\lvert{A \cap
  G}\rvert)=\Theta(x\sqrt{n})$. 

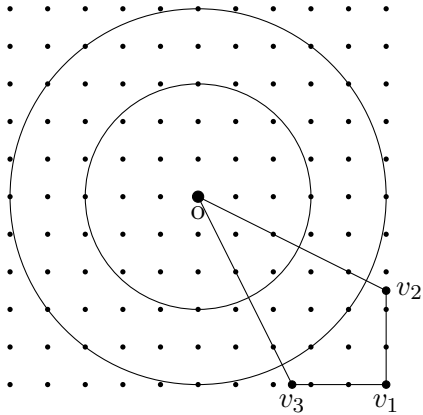
\begin{figure}[ht!]
	\centering
	\begin{tikzpicture}[scale=0.5]
		\foreach \i in {0,1,...,10} 
            \foreach \j in {0,1,...,10} 
                \node [vtx] (v\i\j) at (\i,\j) {};;
		
        \draw (5,5) circle [radius=5]; 
        \draw (5,5) circle [radius=3]; 
        \draw (5,5) -- (7.5,0);
        \draw (5,5) -- (10,2.5);
        \draw (10,2.5) -- (10,0);
        \draw (7.5,0) -- (10,0);
        \node [below] at (10,0) {$v_1$};
        \node [right] at (10,2.5) {$v_2$};
        \node [below] at (7.5,0) {$v_3$};
                \node [draw, shape = circle, fill = black, inner sep=1pt] at (10,0){};
                \node [draw, shape = circle, fill = black, inner sep=1pt] at (10,2.5){};
                \node [draw, shape = circle, fill = black, inner sep=1pt] at (7.5,0){};

		\node [draw, shape = circle, fill = black, inner sep=1.5pt] at (5,5){};
		\node [below] at (v55) {o};

	\end{tikzpicture}

	\caption{The sector $B$ of the annulus $A$. }
	\label{fig:annulus4}
\end{figure}

\begin{lemma}\label{pairs}
The number of descending pairs in $B \cap G$ is $O(x^3\sqrt{n})$.
\end{lemma}
\begin{proof}[Proof of Lemma~\ref{pairs}]
Fix any $u\in B\cap G$, and let $v\in B\cap G$ be such that $(u,v)$ forms a descending pair. By the choice of
the $30^\circ$ sector, the angle $\angle{uvo}$ is at most $60^\circ$, yielding that 
\begin{equation}
||v-o||-||u-o||\geq\frac{||u-v||}{2}.
\end{equation}
Since $u$ and $v$ lie in an annulus of width $x$, we also have $||v-o||-||u-o||\leq x$ and hence
$||u-v||\leq2x$. Therefore, for each fixed $u$, there are $O(x^2)$ possible choices for $v$. As $|B\cap
G|=O(x\sqrt n)$, the number of descending pairs is $O(x^2)\cdot O(x\sqrt n)=O(x^3\sqrt n)$. 
\end{proof}

We can now finalize the proof of the lower bound. Take a random subset $P$ of
$B\cap{G}$ by retaining each point independently with probability
$p\in(0,1)$. We have that
$\mathbb{E}(\lvert{P}\rvert)=\Omega(x\sqrt{n}p)$. Moreover, letting $D$ and $T$
denote the numbers of descending pairs and collinear triples in $P$,
respectively, by Lemmas~\ref{triples} and~\ref{pairs} it holds that 
\begin{equation}
\mathbb{E}(D)=O(p^2x^3\sqrt{n})\quad\text{and}\quad\mathbb{E}(T)=O(p^3x^{5/2}n^{3/4}\log{n}).
\end{equation}
Setting $x=n^{1/10}(\log{n})^{2/5}$, we have $\mathbb{E}(\lvert{P}\rvert)=\Omega\left(pn^{3/5}(\log{n})^{2/5}\right)$ and
\begin{equation}
\mathbb{E}(D)=O\left(p^2n^{4/5}(\log{n})^{6/5}\right)\quad\text{and}\quad\mathbb{E}(T)=O\left(p^3n\log{n}\right).
\end{equation}
Let $p=\frac{c}{n^{1/5}(\log{n})^{4/5}}$ with a sufficiently small constant $c>0$. Then
%We shall have that
\begin{equation}
  \mathbb{E}(\lvert{P}\rvert-D-T)=\mathbb{E}(\lvert{P}\rvert)-\mathbb{E}(D)-\mathbb{E}(T)
  \geq\frac{1}{3}\mathbb{E}(\lvert{P}\rvert)=\Omega\left(\frac{n^{2/5}}{(\log{n})^{2/5}}\right).
\end{equation}
Since each descending pair and collinear triple can be eliminated by deleting
one point, there exists a monotone general position subset of $B\cap G$ of
size $\Omega\left((n/\log{n})^{2/5}\right)$, proving the lower bound. 

\medskip
\emph{Upper bound.} Let $G$ be a $\sqrt{n} \times \sqrt{n}$ grid of $n$
points. A result of Pomerance~\cite{Pom80} states that for every $k \geq 3,b >0$, there exists a number
$m_0(k,b)$ such that if $m \geq m_0(k,b)$ and $\mathbf{u_0}, \mathbf{u_1},\ldots, \mathbf{u_m}$ are points in
$\ZZ^2$ with $\sum_{i=1}^m |\mathbf{u_i} -\mathbf{u_{i-1}}| \leq b m$, then
$\{\mathbf{u_0},\mathbf{u_1},\ldots,\mathbf{u_m}\}$ contains a collinear $k$-tuple. 

Assume for contradiction that there exists an absolute constant $c>0$
such that there exists a monotone general position set
$P:=\{\mathbf{u_0}, \mathbf{u_1},\ldots, \mathbf{u_m}\}\subseteq{G}$ with $m
\geq c \sqrt{n}$. Let $k=3$, $b=2/c$ and assume that $n$ is sufficiently large
so that $m \geq c \sqrt{n} \geq m_0(k,b)$. Since the set $P$ is monotone, we
have 
\begin{equation}
\sum_{i=1}^m |\mathbf{u_i} -\mathbf{u_{i-1}}| \leq 2\sqrt{n} \leq \frac{2m}{c} =b m.
\end{equation}

By the above result (with $k=3$ and $b=2/c$), $P$ contains a collinear triple,
a contradiction. This concludes the proof of the upper bound and the theorem.\qed

\section{Selecting a subset with pairwise distinct slopes} \label{sec:slopes}

In this section we prove Theorem~\ref{thm:distinct-slopes}. 
First we need an upper bound on the number of trapezoids (each trapezoid is an obstacle
in achieving the final goal). A closely related bound was previously obtained by Rudnev and Selig~\cite{RS16}
using the Guth--Katz incidence theorem~\cite{GK15}. Our proof below applies under a slightly more general
non-collinearity assumption and uses only the Szemer\'edi--Trotter theorem~\cite{ST83}. Given the elementary
nature of the argument, it may also be of independent interest.

\begin{lemma}
\label{lem:T}
Let $P\subseteq\mathbb{R}^2$ be a set of $n$ points with at most $s$ points collinear. Then the number of
trapezoids spanned by $P$ is at most 
\begin{equation}
O(n^3 \log{s} + n^2 s^2).
\end{equation}
\end{lemma}
\begin{proof}[Proof of Claim~\ref{lem:T}]
Let $T$ denote the number of trapezoids spanned by $P$. Every trapezoid contains a pair of parallel segments,
which determine two parallel lines. 
Assign every trapezoid to the parallel line that is incident to the largest number of points in $P$,
with ties broken arbitrarily.

For $2 \leq i \leq s$, let $k_i$ be the number of lines containing exactly $i$ points in $P$. A well-known
corollary of the Szemer\'edi--Trotter theorem~\cite{ST83} states that 
\begin{equation} \label{eq:st}
  b_i := \sum_{j \geq i} k_j =O\left( \frac{n^2}{i^3} + \frac{n}{i} \right).
\end{equation}

Note that the sequence $(b_i)_{i\geq2}$ is non-increasing. 
For a given line $\ell$ with $|\ell \cap P|=i$, let $m=m(\ell)$ be the number of
lines determined by $P$ that are parallel to $\ell$ and incident to at most $i$
points of $P$. Let $a_1,\ldots,a_m \leq i$ be the number of points contained in
these lines. Then the number of trapezoids assigned to $\ell$ is at most 
\begin{equation}
{i \choose 2} \sum_{j=1}^{m(\ell)} {a_j \choose 2} \leq {i \choose 2} \frac{n}{i} \, {i \choose 2} \leq ni^3.
\end{equation}
Therefore,
\begin{equation} 
\label{eq:T}
  T \leq \sum_{i=2}^{s}\sum_{\ell:\,|\ell \cap P|=i} ni^3 = n \sum_{i=2}^{s} i^3 k_i.
\end{equation}

Note that $i^3 -(i-1)^3 \leq 3 i^2$, thus by Abel’s partial summation formula we obtain
\begin{equation}
\begin{aligned}
\sum_{i=2}^{s} i^3 k_i &= \sum_{i=2}^{s} i^3 \left(b_i - b_{i+1}\right) 
= O\left(\sum_{i=2}^{s} i^2 b_i \right)
= O\left(\sum_{i=2}^{s} i^2 \left(\frac{n^2}{i^3} + \frac{n}{i} \right) \right)\\
&= O\left( n^2 \sum_{i=2}^{s} \frac{1}{i} + n \sum_{i=2}^{s} i \right)
= O\left( n^2 \log{s} + n s^2 \right).
\end{aligned}
\end{equation}

Plugging the above estimate into~\eqref{eq:T} yields
\begin{equation}
T \leq  n \sum_{i=2}^{s} i^3 k_i =  O\left( n^3 \log{s} + n^2 s^2 \right).\qedhere
\end{equation}
\end{proof}

\subsection{Proof of Theorem~\ref{thm:distinct-slopes}}
\medskip
\emph{Lower bound.} Let $P\subseteq\mathbb{R}^2$ be an arbitrary set of $n$
points with at most $s$ points collinear,
where $s=O(\sqrt{n})$. We take a random subset $X \subseteq P$ by selecting points
independently with probability $p=k/n$. Note that
$\mathbb{E}(\lvert{X}\rvert)=k$.  

To select a subset with distinct slopes, one needs to avoid two types
of obstacles, collinear triples and trapezoids. Let $L$ and $T$ be the number of collinear triples and
trapezoids spanned by $P$.
One obstacle can be eliminated by deleting one point from the subset in the second step. In particular, 
it suffices to choose $k$ so that
\begin{equation}
\mathbb{E}(L p^3) \leq k/3 \quad\text{ and }\quad  \mathbb{E}(T p^4) \leq k/3,
\end{equation}
since this implies that the expected number of remaining points after deleting one point from each trapezoid
and collinear triple is at least $k-k/3-k/3=k/3$. 
Using the upper bounds in Lemma~\ref{low} and  Lemma~\ref{lem:T},
it suffices to choose $k$ so that
\begin{equation}
O(k^2) \leq n/\log{s} \quad\text{ and }\quad  O(k^3) \leq n/\log{s}.
\end{equation}
Setting $k = c(n/\log{s})^{1/3}$ with a sufficiently small constant $c>0$ satisfies both requirements,
proving the lower bound. 
 
\medskip
\emph{Upper bound (i).} Let $P=\{(x,x^2): x\in [n]\}\subseteq \mathbb{R}^2$.
The point set $P$ is of size $n$ and contains no collinear triple.
Let $Q\subseteq P$ be a subset of size $|Q|=2\sqrt{n}+2$. We show that $Q$ contains a trapezoid. 

A \emph{Sidon set} $S\subseteq [n]$ is a subset of the integers such that
$a+b = c+d$ for $a, b, c, d \in S$ with $a\neq b, c\neq d$ implies $\{a, b\} = \{c, d\}$;
see, \eg,~\cite[Chap.~6]{ES03} or~\cite[Chap.~2]{TV06}. 
Note that the largest size of a Sidon set $S\subseteq[n]$ is at most $2\sqrt{n}+1$,
because $\binom{|S|}{2}\leq 2n$. Let $S=\{x: (x,x^2)\in Q\}$.
Since $|S|=2\sqrt{n}+2$, it is not a Sidon set and thus contains 4 distinct elements $x_1,x_2,x_3,x_4\in S$
such that $x_1+x_2=x_3+x_4$. Then, the 4 points 
$(x_1,x_1^2),  (x_2,x_2^2),  (x_3,x_3^2), (x_4,x_4^2)\in Q $
form a trapezoid. Indeed,
\begin{equation}
\frac{x_2^2-x_1^2}{x_2-x_1}=  
x_1+x_2  = x_3+x_4 = \frac{x_4^2-x_3^2}{x_4-x_3}.
\end{equation}

\medskip
\emph{Upper bound (ii).} The proof is similar to that of Theorem~\ref{law} part~(ii).
Consider the 2-dimensional grid $[n]^2$. Let $\mathcal{L}$ be the set of
lines that go through at least two points in $[n]^{2}$. Note that
$\lvert\mathcal{L}\rvert\leq{n^{4}}$ and every line in $\mathcal{L}$ contains at
most $n$ points in $[n]^2$. Moreover, it was shown in~\cite{EGRT92} that every
set of $\Omega(n^{4/5})$ points in $[n]^2$ contains a trapezoid. Let $P$ be a
random subset of $[n]^2$ by retaining each point independently with probability
$n^{\alpha-1}/2$, where $\alpha \in [3/5,1]$. By Chernoff's bound we have that
$|P|=\Theta(n^{1+\alpha})$ with high probability. Furthermore, using the union
bound and Chernoff's bound we have 
\begin{equation}
\begin{aligned}
  & \quad \ \mathbb{P}\left(P\text{ contains no collinear
  $n^{\alpha}$-tuple}\right)=1-\mathbb{P}\left(\exists\,\ell\in\mathcal{L}:\,\lvert{P\cap\ell}\rvert\geq{n^{\alpha}}\right) \\
  &\geq1-\sum_{\ell\in\mathcal{L}}\mathbb{P}\left(\lvert{P\cap\ell}\rvert
  \geq\left(1+\frac{n^{\alpha}}{2\mathbb{E}(\lvert{P\cap\ell}\rvert)}\right)\mathbb{E}(\lvert{P\cap\ell}\rvert)\right)
  \geq1-n^{4}\exp(-n^{\alpha}/10),
\end{aligned}
\end{equation}
which tends to $1$ as $n\to\infty$. Namely, there exists a point set $P\subseteq[n]^{2}$
of size $\Theta(n^{1+\alpha})$, so that $P$ contains no collinear $n^{\alpha}$-tuple
and every subset of $P$ of size $\Omega(n^{4/5})$ contains a trapezoid. After scaling we obtain 
\begin{equation} 
h(n,n^{\frac{\alpha}{1+\alpha}}) = O(n^{\frac{4/5}{1+\alpha}}) \quad \text{for}\quad 3/5\leq\alpha\leq1.
\end{equation}
Note that $\frac38 \leq \frac{\alpha}{1+\alpha} \leq \frac12$ in this interval.
By letting $s=n^{\frac{\alpha}{1+\alpha}}$ we have
\begin{equation}
h(n,s) = O((n/s)^{4/5}) \quad \text{for}\quad n^{3/8}\leq s \leq \sqrt{n}.
\end{equation}

\medskip
Aggregating the bounds in (i) and (ii) yields the bounds in the theorem.\qed

\section{Further consequences}\label{sec:consequences}

Gowers~\cite{Gow11} introduced the following natural notion. For
$s\in\mathbb{N}$, the \emph{(diagonal) Ramsey number of planar points} $r(s)$
is the smallest integer, such that every set $P\subseteq\mathbb{R}^{2}$ of $r(s)$
points contains either a collinear $s$-tuple or a general position subset of
size $s$. It follows from the definition that $r(s)=s$ when $s\leq3$. With a
simple drawing one can show that $r(4)=5$, but the value of $r(5)$ is already
unknown. For the asymptotic behavior of $r(s)$, Gowers~\cite{Gow11} observed
that 
\begin{equation}
\label{gow}
\Omega(s^{2})=r(s)=O(s^{3}).
\end{equation}
The lower bound construction comes from the grid $[m]^{2}$, where
$m=\lceil{s/2}\rceil-1$. Obviously, $[m]^{2}$ contains at most $s/2$ collinear
points and by the pigeonhole principle at most $s-1$ points in general
position. The upper bound in~\eqref{gow} is derived by a greedy algorithm. Gowers~\cite{Gow11} asked to
determine the correct order of $r(s)$ or 
at least to show which one of the trivial lower and upper bounds is closer to
the truth. This question has been partially answered by Payne and
Wood~\cite{PW13}, who proved that 
\begin{equation}
\label{pop}
r(s)=O(s^{2}\log{s}),
\end{equation}
and further conjectured that $r(s)$ is quadratic in $s$.
\begin{conjecture}[{Payne--Wood~\cite[Conjecture 4.1]{PW13}}, see also~\cite{Gow13}]
\label{oil}
\begin{equation}
r(s)=\Theta(s^{2}).
\end{equation}
\end{conjecture}
As a consequence of Proposition~\ref{new}, if a set $P$ of
$n$ points in the plane contains no collinear $s$-tuple with $s=O(\sqrt{n})$, then it contains a general
position subset of size $\Omega\left(\sqrt{\frac{n\log\log{s}}{\log{s}}}\right)$. 
This yields the current best upper bound on $r(s)$, first proved by Hajnal and Szemer\'{e}di~\cite{HS18}. 
\begin{corollary}[{Hajnal--Szemer\'{e}di~\cite[Theorem 3]{HS18}}]
\label{log}
\begin{equation}
r(s)=O\left(\frac{s^{2}\log{s}}{\log\log{s}}\right).
\end{equation}
\end{corollary}
Our proof of Proposition~\ref{new} can be regarded as an extension of the arguments presented in~\cite{HS18}.
Payne and Wood~\cite{PW13} proposed the following strengthening of Conjecture~\ref{oil}.
\begin{conjecture}[{Payne--Wood~\cite[Conjecture 4.2]{PW13}}]
\label{conj:PaynWood}
  Every set $P\subseteq\mathbb{R}^{2}$ of $n$ points with at most $\sqrt{n}$ points collinear can be colored
  with $O(\sqrt{n})$ colors such that each color class is in general position. 
\end{conjecture}
Conjecture~\ref{conj:PaynWood}, if true, would be sharp since at least $\sqrt{n}/2$ colors are needed to
properly color the $\sqrt{n}\times \sqrt{n}$ grid. 
Payne and Wood showed that $O(\sqrt{n}(\log{n})^{3/2})$ colors are sufficient to color $P$.
Here we improve this result by a factor of $\sqrt{\log{n}}$. Moreover, note that Theorem~\ref{dog} is a
strengthening of Corollary~\ref{log}. 
\begin{theorem}
\label{dog}
Every set $P\subseteq\mathbb{R}^{2}$ of $n$ points with at most $\sqrt{n}$
collinear points can be colored with $O(\sqrt{n}\log{n})$ colors such that each
color class is in general position. In addition, one can properly color $n-o(n)$ points of $P$ with
$O(\sqrt{n\log{n}/\log\log{n}})$ colors.  
\end{theorem}
\begin{proof}[Proof of Theorem~\ref{dog}]
Assume that $n$ is sufficiently large and let $s=\lfloor\sqrt{n}\rfloor$. Then
$P\subseteq\mathbb{R}^{2}$ is a set of $n$ points with at most $s$ points
collinear. Our process of coloring $P$ consists of two phases.

\smallskip
\textbf{Phase 1:}\\
Let $k=\lfloor\log\log{n}-2C\rfloor$ with some sufficiently large constant
$C>0$. For each $i\in[k]$, we let $P_{i}$ be a subset of
$P\backslash\left(\bigcup_{j<i}P_{j}\right)$ obtained by selecting half of the remaining points,
namely those that are contained in the least number of collinear triples.
Note that $\lvert{P_{i}}\rvert\geq4^Cn/\log{n}$ holds for all $i\in[k]$, namely,
$s\leq2^{-C}\sqrt{\lvert{P_{i}}\rvert\log{\lvert{P_{i}}\rvert}}$. Let
$\mathcal{H}_{i}=\mathcal{H}(P_{i})$ be a $3$-uniform hypergraph with the vertex
set $P_{i}$, whose edges are the collinear triples in $P_{i}$. Observe that the
least number of colors needed to color $P_{i}$ with each color class being in
general position is equal to $\chi(\mathcal{H}_{i})$. Recall~\eqref{due}, we
have that 
\begin{equation}
  \chi(\mathcal{H}_{i})=O\left(\sqrt{\frac{\lvert{P}_{i}\rvert\log{s}}{\log{\frac{\lvert{P}_{i}\rvert\log{s}}{s^{2}}}}}\right)
  =O\left(\sqrt{\frac{n\log{s}}{2^{i}\log{\frac{n\log{s}}{2^{i}s^{2}}}}}\right)
  =O\left(\sqrt{\frac{n\log{n}}{2^{i-1}\log\log{n}}}\right).
\end{equation}
Hence, we can properly color each $P_{i}$ using $\chi(\mathcal{H}_{i})$ new colors for
$1\leq i\leq k$. The total number of colors to properly color $\bigcup_{i=1}^{k}P_{i}$ is at most
\begin{equation}
  \sum_{i=1}^{k}O\left(\sqrt{\frac{n\log{n}}{2^{i-1}\log\log{n}}}\right)
  \leq O\left(\sqrt{\frac{n\log{n}}{\log\log{n}}}\right)\sum_{i=1}^{\log\log{n}}\left(\frac{1}{\sqrt{2}}\right)^{i-1}
  =O\left(\sqrt{\frac{n\log{n}}{\log\log{n}}}\right).
\end{equation}

\smallskip
\textbf{Phase 2:}\\
Let $P'$ be the set of remaining uncolored points. It holds that
$\lvert{P'}\rvert=O\left(n/\log{n}\right)$. We will color $P'$ as
follows. First, select a largest general position subset $I_{1}$ of $P'$ and let
$P'_{1}=P\backslash{I_{1}}$. For $i>1$, as long as $P'_{i-1}\neq\emptyset$, we
select a largest general position subset $I_{i}$ of $P'_{i-1}$ and let
$P'_{i}=P'_{i-1}\backslash{I_{i}}$. Let $m$ be the stopping time of this
process, i.e., $P'_{m}=\emptyset$, equivalently, 
\begin{equation}
P'=\bigcup_{i=1}^{m}I_{i}.
\end{equation}
By assigning a new color to each $I_{i}$, we finish coloring $P'$ with $m$
colors. It remains to bound the value of $m$. Let $P'_{0}=P'$.
By~\eqref{pay} we have that, for each $i\in[m]$,
\begin{equation}
\lvert{I_{i}}\rvert=\Omega\left(\frac{\lvert{P'_{i-1}}\rvert}{s}\right)\geq\frac{\lvert{P'_{i-1}}\rvert}{D\sqrt{n}},
\end{equation}
for some constant $D>0$. Accordingly, for each $i\in[m]$,
\begin{equation}
\lvert{P'_{i}}\rvert\leq\lvert{P'}\rvert\left(1-\frac{1}{D\sqrt{n}}\right)^{i}.
\end{equation}
When $m=\lceil{D\sqrt{n}\log{n}}\rceil$, it holds that
$\lvert{P'_{m-1}}\rvert\leq2$. Then, we can choose $I_{m}$ to be $P'_{m-1}$ and
the process stops, namely, $m=O(\sqrt{n}\log{n})$. Overall, we have properly colored $P$ using 
\begin{equation}
O\left(\sqrt{n\log{n}/\log\log{n}}\right)+O\left(\sqrt{n}\log{n}\right)=O\left(\sqrt{n}\log{n}\right)
\end{equation}
colors.
\end{proof}

Theorem~\ref{dog} guarantees that most of the points from $P$ can be properly colored with
\begin{equation}
O\left(\sqrt{n\log{n}/\log\log{n}}\right)    
\end{equation}
colors. This suggests that extending the result to all points may be achievable with further refinements in the proof.

\section*{Acknowledgements}
The first author is supported in part by NSF grants DMS-1764123 and RTG
DMS-1937241, FRG DMS-2152488 and the Arnold O. Beckman Research Award (UIUC Campus Research Board RB 24012), and the Simons Fellowship. The research of the second author is supported by the Institute for Basic Science (IBS-R029-C4). The authors thank the anonymous referees for their useful comments on the manuscript.

\section*{Appendix: A matching lower bound for Lemma~\ref{triples}}
\label{sec:appendix}
The estimate in Lemma~\ref{triples} is tight up to a constant factor. Recall that we assume $\sqrt{n}=2m+1$
and let $x=x(n)\geq1$ with 
$x\ll\sqrt{n}$. Let $G$ be a $\sqrt{n}\times\sqrt{n}$ grid, and let $o$ be the center point of $G$. Let $A$ be
the annulus formed by two concentric circles of radii $m$ and $m-x$, 
centered at $o$. Our goal is to show that the number of collinear triples in $A\cap G$ is $\Omega(x^{5/2}n^{3/4}\log{n})$.

After a translation, assume that $G=[-m,m]^2\cap\mathbb{Z}^2$ and $o=(0,0)$. Fix coprime integers $q$ and $r$
with $1\leq q\leq r\leq c\sqrt{mx}$, where $c>0$ is a sufficiently small constant. 

By the first equation in~\eqref{eq:lines_distance}, consecutive lattice lines of slope $(q,r)$ are at distance
$1/\sqrt{q^2+r^2}=\Theta(1/r)$. Thus, $\Omega(xr)$ such lines have distance from $o$ in $[m-2x,m-x]$. For any
such line $\ell$, write $\dist(\ell,o)=m-tx$, where $1\leq t\leq2$. Each component of $\ell\cap A$ has length 
\begin{equation}
\sqrt{m^2-(m-tx)^2}-\sqrt{(m-x)^2-(m-tx)^2}=\Omega(\sqrt{mx}),
\end{equation}
where we used $x\ll m$. Since consecutive lattice points on $\ell$ are at distance $\sqrt{q^2+r^2}=\Theta(r)$,
the line $\ell$ contains $\Omega(\sqrt{mx}/r)$ points of $A\cap G$. Choosing $c$ sufficiently small, this
quantity is at least three, and therefore the number $T(q,r)$ of collinear triples contained in lines of slope
$(q,r)$ is 
\begin{equation}
T(q,r)=\Omega\left(xr\left(\frac{\sqrt{mx}}{r}\right)^3\right)
=\Omega\left(\frac{x^{5/2}m^{3/2}}{r^2}\right).
\end{equation}
Let $\phi(r)$ denote Euler's totient function, that is, $\phi(r)$ counts the number of positive integers
smaller than and coprime to $r$. Using $\sum_{r\leq R}\phi(r)/r^2=\Theta(\log{R})$
(see, e.g.,~\cite[Theorem 3.2]{A13}), we obtain 
\begin{equation}
\sum_{\substack{1\leq r\leq c\sqrt{mx}\\1\leq q\leq r,(q,r)=1}}T(q,r)
=\Omega\left(x^{5/2}m^{3/2}\log(mx)\right)=\Omega\left(x^{5/2}n^{3/4}\log{n}\right).
\end{equation}
Therefore, the number of collinear triples in $A\cap G$ is $\Omega\left(x^{5/2}n^{3/4}\log{n}\right)$.
\end{document}